\documentclass{article}
\usepackage{amsmath}
\usepackage{ntheorem-hyper}
\usepackage{amssymb}
\usepackage{authblk} % author envir package
\usepackage{algorithm}
\usepackage{algorithmicx}
\usepackage{algpseudocode}
\usepackage{graphicx}
\usepackage{graphbox}
\usepackage{subfigure}
\usepackage{svg}
\usepackage{epstopdf}
\usepackage{tabularx}  
\usepackage{booktabs}
\usepackage[numbers, sort]{natbib}
\usepackage{hyperref}
\usepackage[top=2cm, bottom=2cm, left=2cm, right=2cm]{geometry}  
\usepackage{xcolor} % \textcolor{red/blue/green/black/white/cyan/magenta/yellow/teal}{text}
\usepackage{color}
\usepackage[title]{appendix}

\usepackage[T1,T2A]{fontenc}
\usepackage[utf8]{inputenc}
\usepackage[russian, german, english]{babel}

\hypersetup{urlcolor=blue, citecolor=red}

\newtheorem{theorem}{Theorem}

\newtheorem*{proof}{Proof}

\def\RR{\mathbb{R}}

\def\ZZ{\mathbb{Z}}

\title{A flat geodesic with a periodic asymptote}
\author[+]{Hongda Qiu}
\begin{document}
\date{}
\maketitle

\begin{abstract}
We construct a closed surface of genus $2$ with nonpositive Gaussian curvature
and a complete nonperiodic flat geodesic asymptotic to a flat closed geodesic.
The metric is of class $C^{3,1/3}$ and its curvature is of class $C^{1,1/3}$,
but not $C^2$. Thus the periodic-asymptote obstruction for smooth surfaces
does not extend to $C^1$ Gaussian curvature.
\end{abstract}

\textit{Introduction.}
A geodesic is called \emph{flat} if the Gaussian curvature vanishes along it.
For a smooth closed surface of genus at least $2$ and nonpositive curvature,
the ergodicity problem studied by Pesin~\cite{Pesin1977} leads to the question
whether the unit vectors tangent to complete flat geodesics form a set of
zero Liouville measure; an affirmative answer implies ergodicity of the
geodesic flow. A stronger question asks whether every complete flat geodesic
is periodic; see \cite[Section~1]{Wu2015} for both questions and their relation.
One restriction in the smooth case is that a nonperiodic flat geodesic cannot
have a periodic asymptote~\cite[pp.~23, 29--30]{Petrunin}.
We show that this restriction fails for a metric of class $C^{3,1/3}$,
with Gaussian curvature of class $C^{1,1/3}$, on a closed surface of genus $2$.
The example concerns lower regularity: it does not settle the measure-zero
problem or the periodicity question for smooth metrics.

\begin{theorem}\label{thm:example}
There is a closed orientable surface $M$ of genus $2$ with a metric
$g\in C^{3,1/3}$ and Gaussian curvature $K\leq 0$ which contains a complete
nonperiodic flat geodesic $\gamma$ with a flat periodic asymptote $\delta$.
In particular, $K\in C^{1,1/3}$.
The parametrizations can be chosen so that
\[
 d_g\bigl(\gamma(t),\delta(t)\bigr)=O(t^{-1/3})
 \qquad\text{as }t\to\infty.
\]
\end{theorem}

\begin{proof}\normalfont
The idea is to glue the sides of a narrowing strip with a shift.
Its missing end will be a closed geodesic.
We arrange that the other end is Euclidean; this allows us to close the surface.

\textit{The functions.}
Choose an even nonnegative function $\rho\in C^\infty_c((-1,1))$ with
$\int_{-1}^{1}\rho(s)\,ds=1$. Put
\[
 f_0(t)=
 \begin{cases}
  1-t,&t\leq 0,\\
  (1+t/2)^{-2},&t\geq 0,
 \end{cases}
 \qquad
 f(t)=\frac12\int_{-1}^{1}f_0(2t-1-s)\rho(s)\,ds.
\]
Then $f$ is smooth, positive, decreasing, and convex. Moreover,
\begin{equation}\label{eq:profile}
 f(s)=1-s\quad(s\leq 0),\qquad
 f(s)\asymp(1+s)^{-2},\qquad
 |f^{(m)}(s)|\leq C_m(1+s)^{-2-m}\quad(s\geq 0).
\end{equation}
Here $a\asymp b$ means that $a/b$ is bounded above and below by positive
constants. Constants in the proof need not be the same at different occurrences.

For $\theta\in(0,\pi)\cup(\pi,2\pi)$, set
\[
 w(\theta)=|\sin\theta|^{-3/7}-\frac12\sin\theta,
 \qquad
 u(\theta)=\int_0^\theta w(s)\,ds.
\]
The singularities of $w$ are integrable; see Figure~\ref{fig:w-theta}.
Let $\theta(u)$ be the inverse function.

\begin{figure}[htbp]
\centering
\includegraphics[width=0.82\textwidth]{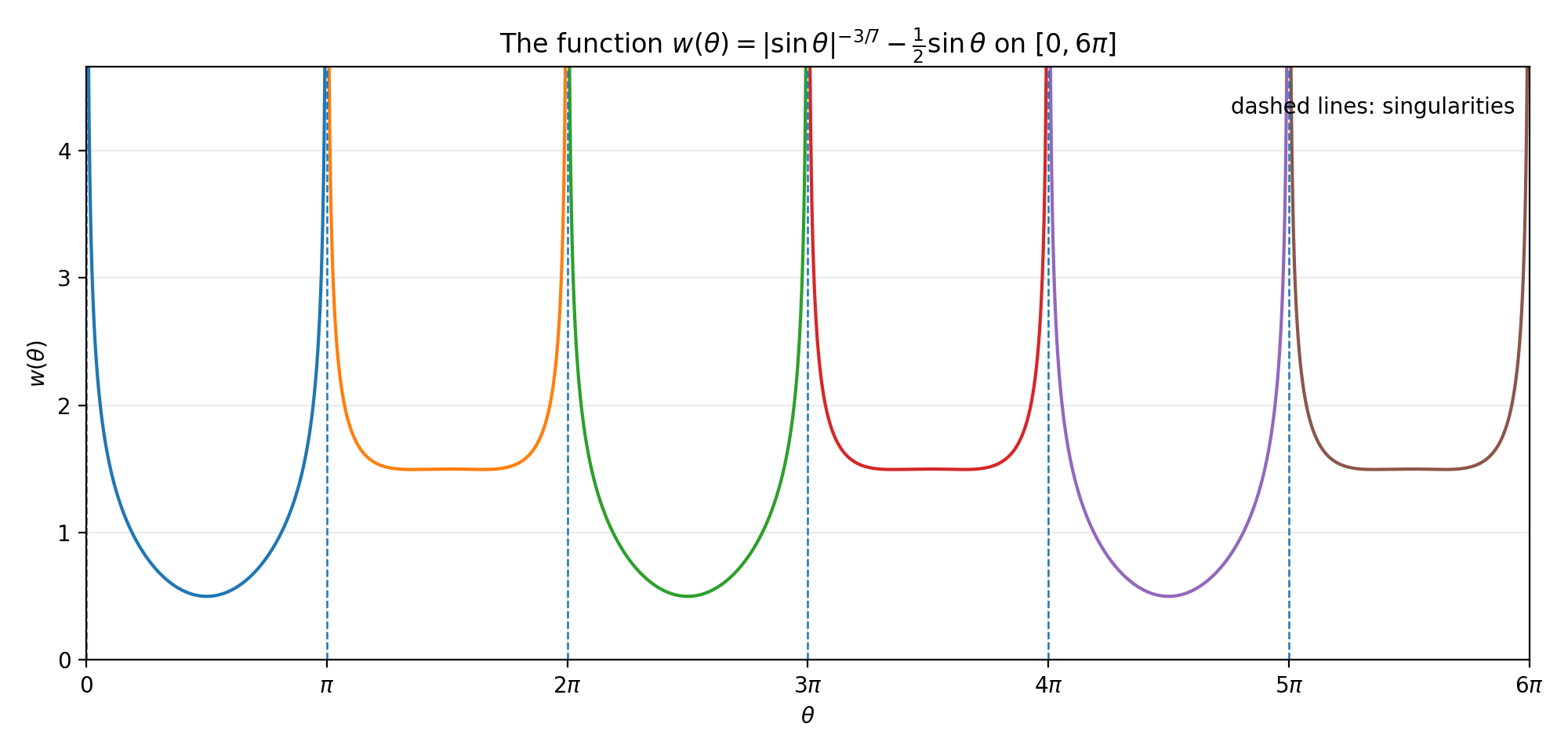}
\caption{The function $w(\theta)=|\sin\theta|^{-3/7}-\frac12\sin\theta$ on $[0,6\pi]$.
The dashed lines mark its singularities at $k\pi$, $k=0,1,\dots,6$.}
\label{fig:w-theta}
\end{figure}
Set $h(u)=\theta'(u)$, with value zero at $0,u(\pi),u(2\pi)$.
This function is positive and smooth elsewhere. Near each of its zeros $u_i$,
\begin{equation}\label{eq:h}
 h(u)\asymp|u-u_i|^{3/4},\qquad
 |h'(u)|\leq C h(u)^{-1/3},\qquad
 |h''(u)|\leq C h(u)^{-5/3}.
\end{equation}
Indeed, $|u-u_i|\asymp|\theta-\theta_i|^{4/7}$. Also,
\begin{equation}\label{eq:moments}
 \int_0^{u(2\pi)} h(s)\,ds=2\pi,\qquad
 \int_0^{u(2\pi)}(\cos\theta(s),\sin\theta(s))\,ds=\left(0,-\frac\pi2\right).
\end{equation}
The second identity follows by changing variables to $\theta$ and using symmetry.

\textit{The strip and its missing circle.}
On $\RR\times[0,u(2\pi)]$, put
\begin{equation}\label{eq:strip}
 g=dt^2+G(t,u)^2\,du^2,\qquad G(t,u)=f(th(u)),
\end{equation}
and identify
\begin{equation}\label{eq:gluing}
 (t,0)\sim(t+\pi/2,u(2\pi)).
\end{equation}
The boundary metrics agree. The quotient is an open annulus, denoted by $Q$.
For now, \eqref{eq:strip} defines a length metric; its differentiability will be
checked in different coordinates.

Where $h>0$, the usual curvature formula gives
\begin{equation}\label{eq:K}
 K=-\frac{G_{tt}}G=-h^2\frac{f''(th)}{f(th)}\leq 0.
\end{equation}
Give $K$ the value zero on the lines $h=0$.
Since $f''$ vanishes to all orders at zero, this causes no loss of
regularity at any finite value of $t$.

Let $\ell(t)=\int_0^{u(2\pi)}G(t,s)\,ds$.
For $t\geq1$, \eqref{eq:profile} and \eqref{eq:h} imply
\begin{equation}\label{eq:ell}
 \ell(t)\asymp t^{-4/3},\qquad
 |\ell^{(m)}(t)|\leq C_m t^{-4/3-m}.
\end{equation}
For the first estimate, near a zero of $h$ use
\[
 \int_0^\varepsilon(1+t s^{3/4})^{-2}\,ds\asymp t^{-4/3}.
\]
The contribution away from the zeros is $O(t^{-2})$.
The derivative estimates follow in the same way.

Define
\begin{equation}\label{eq:coordinates}
 r(t)=\sum_{n=1}^{\infty}\ell(t+n\pi/2),\qquad
 x=t\pmod{\pi/2},\qquad
 R(t,u)=r(t)+\int_0^u G(t,s)\,ds.
\end{equation}
The series and its derivatives converge locally uniformly.
Since $r(t)=r(t+\pi/2)+\ell(t+\pi/2)$, these coordinates agree under
\eqref{eq:gluing}. For fixed $x$, the successive intervals
$[r(t),r(t-\pi/2)]$ fit together and cover $(0,\infty)$.
Thus $(x,R)$ identify $Q$ with $(\RR/(\tfrac\pi2\ZZ))\times(0,\infty)$.
Furthermore,
\begin{equation}\label{eq:radial-rate}
 R(t,u)\asymp r(t)\asymp t^{-1/3}\qquad(t\to\infty),
\end{equation}
uniformly in $u$.

Subscripts $t$ denote derivatives with $u$ fixed; $\partial_x$ and $\partial_R$
act in the new coordinates. Since $dR=R_t\,dt+G\,du$, we have
\begin{equation}\label{eq:metric}
 g=dx^2+(dR-R_t\,dx)^2.
\end{equation}
Moreover,
\begin{equation}\label{eq:radial-derivatives}
 \begin{gathered}
 \partial_R R_t=\frac{G_t}{G},\qquad
 \partial_x R_t=R_{tt}-R_t\frac{G_t}{G},\\
 R_t=O(t^{-4/3}),\qquad
 \frac{G_t}{G}=O(t^{-1}),\qquad R_{tt}=O(t^{-7/3}).
 \end{gathered}
\end{equation}
These identities agree across the cut.
It follows that $R_t$ extends to $R=0$ as a $C^1$ function with
$R_t=d(R_t)=0$ there.
Add the circle $\delta=\{R=0\}$.
It has length $\pi/2$.

Since \eqref{eq:metric} is at least $dx^2$, the curves
\[
 \gamma_0(t)=[t,0],\qquad \gamma_1(t)=[t,u(\pi)],\qquad
 \delta(t)=(t\pmod{\pi/2},0)
\]
locally minimize length and are unit-speed geodesics.
For $\gamma_0$, the radial segment in \eqref{eq:coordinates} gives
\begin{equation}\label{eq:asymptote}
 d_g\bigl(\gamma_0(t),\delta(t)\bigr)
 \leq r(t)=O(t^{-1/3}).
\end{equation}
Both $\gamma_0$ and $\gamma_1$ approach $\delta$ as $t\to\infty$.

\textit{Regularity.}
It remains to check the metric and the curvature at $R=0$.
Put
\[
 p=\frac{f'}f,\qquad q=\frac{f''}f,\qquad
 H(s)=\frac{2q(s)+s q'(s)}{f(s)}.
\]
Differentiating the defining integral for $f$ gives, for $s\geq0$,
\begin{equation}\label{eq:auxiliary}
 \begin{gathered}
 |q^{(m)}(s)|\leq C_m(1+s)^{-2-m}\quad(0\leq m\leq2),\\
 |H(s)|\leq C(1+s)^{-1},\qquad
 |H'(s)|\leq C(1+s)^{-2},\qquad
 \left|\frac{p(s)+s p'(s)}{f(s)}\right|\leq C.
 \end{gathered}
\end{equation}
For example, $f(s)=\frac12s^{-2}(1-s^{-1}+O(s^{-2}))$, with differentiated
remainders. This explains the cancellation in $H$.
The functions $q$ and $H$ vanish to all orders at zero.

In the $(x,R)$ coordinates, \eqref{eq:K} yields
\begin{equation}\label{eq:K-derivatives}
 K_R=-hh'H(th),\qquad
 K_x=-h^3q'(th)-R_t K_R.
\end{equation}
Thus
\[
 |K|\leq Ct^{-2},\qquad
 |K_R|\leq C\frac{h^{2/3}}{1+th}\leq Ct^{-2/3},\qquad
 |K_x|\leq Ct^{-2}.
\]
Together with \eqref{eq:radial-rate}, this proves
\begin{equation}\label{eq:K-extension}
 K=O(R^6),\qquad K_R=O(R^2),\qquad K_x=O(R^6).
\end{equation}
In particular, $K$ extends as a $C^1$ function with $K=dK=0$ on $\delta$.

We also need a H\"older estimate to recover the differentiability of the metric.
One more differentiation of \eqref{eq:K-derivatives}, using
\eqref{eq:h} and \eqref{eq:auxiliary}, gives
\begin{equation}\label{eq:second}
 |K_{xx}|+|K_{xR}|+|K_{RR}|\leq Ct\leq CR^{-3}\qquad(t\geq1).
\end{equation}
For example,
\[
 K_{RR}=-\frac{((h')^2+hh'')H(th)+th(h')^2H'(th)}{f(th)}.
\]
It is $O(t)$; the other derivatives satisfy the same bound.
The derivatives extend across each finite-$t$ cut because $q$ is flat at zero.
If two points are a coordinate distance $d<R/2$ apart, then
\[
 |K_x(z)-K_x(z')|+|K_R(z)-K_R(z')|\leq C\min\{R^2,R^{-3}d\}\leq C d^{2/5}.
\]
For the remaining pairs use \eqref{eq:K-extension}.
Consequently, $K\in C^{1,2/5}$, and in particular $K\in C^{1,1/3}$.

Next, $G_t/G=hp(th)$ is uniformly $3/4$-H\"older in $R$.
Indeed, where $h\leq1/t$, the measures $dR$ and $du$ are comparable,
and \eqref{eq:h} applies directly.
Where $h\geq1/t$, we have
\[
 \left|\frac{G_t}{G}\right|\leq Ct^{-1},\qquad
 \left|\partial_R\left(\frac{G_t}{G}\right)\right|
 =\left|h'\frac{p(th)+thp'(th)}{f(th)}\right|\leq Ct^{1/3}.
\]
Hence $\left|\frac{G_t}{G}(z)-\frac{G_t}{G}(z')\right|\leq C\min\{t^{-1},t^{1/3}d\}\leq Cd^{3/4}$
on each such interval.
There are uniformly finitely many intervals in each strip.
The estimates also join across strips, since $G_t/G=0$ on their common edges.

For $V=\partial_x+R_t\partial_R$, we have
\begin{equation}\label{eq:curvature-identity}
 V\left(\frac{G_t}{G}\right)=-K-\left(\frac{G_t}{G}\right)^2.
\end{equation}
Following $V$ and then a radial segment proves the same H\"older estimate
in both coordinates.
Also, $\partial_R R_{tt}=-K$ and $\partial_x R_{tt}=R_{ttt}+R_tK$ are bounded.
Equation \eqref{eq:radial-derivatives} now gives $R_t\in C^{1,3/4}$.
Thus the metric \eqref{eq:metric} is at least $C^{1,1/3}$ up to $\delta$.

Reflect the metric across $R=0$, extending $R_t$ oddly and $K$ evenly.
The preceding estimates remain valid.
There is no additional curvature on the cuts or on $\delta$:
\eqref{eq:curvature-identity} can be written as
\[
 \partial_x\left(\frac{G_t}{G}\right)
 +\partial_R\left(R_t\frac{G_t}{G}\right)=-K,
\]
and integration across a cut gives no boundary term; at $R=0$ use $R_tG_t/G\to0$.

Finally, use isothermal coordinates, which are $C^{2,1/3}$ functions of
$(x,R)$; see \cite[Section~1]{DeTurckKazdan}.
In these coordinates $g=e^{2\varphi}(dy_1^2+dy_2^2)$ and
\[
 \Delta\varphi=-K e^{2\varphi}.
\]
Initially $\varphi\in C^{1,1/3}$. The equation holds weakly, so elliptic
regularity gives $\varphi\in C^{3,1/3}$.
The transitions between oriented isothermal charts are holomorphic.
They therefore give a smooth structure in which $g\in C^{3,1/3}$.
In particular, \eqref{eq:K} is the classical Gaussian curvature, including
on the flat geodesics and the reflected circle.

\textit{Closing the surface.}
Set
\[
 E(u)=(\cos\theta(u),\sin\theta(u)),\qquad
 N(u)=(-\sin\theta(u),\cos\theta(u)),\qquad
 c(u)=\int_0^u E(s)\,ds.
\]
For $t\leq0$, we have $G=1-th$.
Since $N'=-hE$, the map
\begin{equation}\label{eq:Euclidean-end}
 \Psi(t,u)=c(u)+tN(u)
\end{equation}
satisfies $\Psi_t=N$ and $\Psi_u=(1-th)E$.
It is therefore a local isometry to the Euclidean plane.
By \eqref{eq:moments}, $c(u(2\pi))=-\tfrac\pi2N(0)$ and $N(u(2\pi))=N(0)$.
Hence $\Psi(t+\pi/2,u(2\pi))=\Psi(t,0)$, as required by the gluing.

Far enough out, $\Psi$ is one-to-one and covers the exterior of a disk.
Here is a direct check.
For a sufficiently distant point $z$, a preimage with $t<0$ must satisfy
\[
 \langle z-c(u),E(u)\rangle=0,\qquad
 t=\langle z-c(u),N(u)\rangle<0.
\]
All such solutions lie in a small interval about the direction
$N(u)=-z/|z|$.
In the $\theta$ variable, the derivative of the first expression is
$t-w(\theta)<0$ there, and its endpoint values have opposite signs.
Thus there is exactly one solution, with the boundary identification understood.

Symmetry also gives $c(u(\pi))=(0,b)$ for some $b$.
Consequently,
\[
 \Psi(t,0)=(0,t),\qquad \Psi(t,u(\pi))=(0,b-t).
\]
The two flat geodesics therefore leave through opposite rays of the same
Euclidean line.

Take a large flat square torus and remove a round disk centered on a vertical
closed geodesic. Replace the disk by the part of our annulus between
$\delta$ and a large circle in its Euclidean end.
The metrics agree in Euclidean neighborhoods of the gluing circle.
The result is a torus with one boundary component, namely $\delta$.

The vertical torus geodesic outside the removed disk joins the two flat rays.
Together they form a complete flat geodesic $\gamma$ whose two ends approach
$\delta$. It is not periodic, since it contains points outside $\delta$ but
approaches $\delta$.
Double the surface along $\delta$.
The result is a closed orientable surface of genus $2$.
See Figure~\ref{fig:topological-model}.
The regularity and the sign of its curvature were checked above, and
\eqref{eq:asymptote} gives the stated rate after a translation of the parameter.
\hfill$\square$
\end{proof}

\begin{figure}[htbp]
\centering
% Crop the preview to the surface; its legend is described in the caption.
\includegraphics[width=0.43\textwidth,
 trim=82bp 129.2bp 69.2bp 119.2bp,clip]{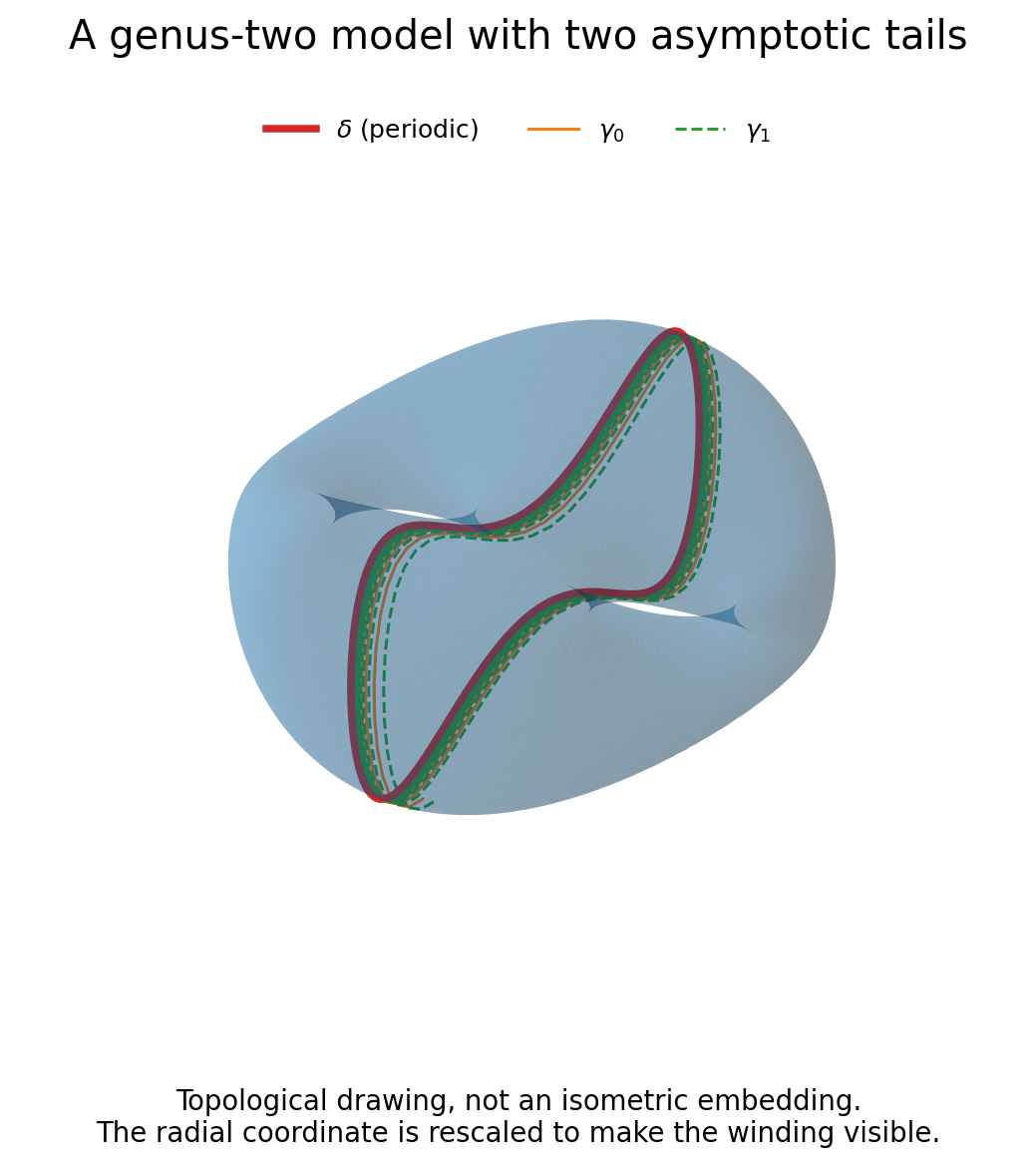}
\caption{A topological model of $M$, not an isometric embedding.
The red curve is $\delta$; the orange and dashed green curves show the tails
$\gamma_0$ and $\gamma_1$. Their connecting arc is omitted;
transverse distances are rescaled.}
\label{fig:topological-model}
\end{figure}

\textit{Remark.}
The curvature is not $C^2$.
Fix $u$ with $h(u)>0$. Then $K(t,u)\sim-6t^{-2}$, while the transverse arc
from $(t,u)$ to the flat line has length at most $\ell(t)\asymp t^{-4/3}$.
If $K$ were $C^2$, compactness and $K\leq0$ would give
$|K(t,u)|\leq C\ell(t)^2=O(t^{-8/3})$, a contradiction.
This is the estimate used in the periodic-asymptote obstruction in
\cite[pp.~29--30]{Petrunin}.

\begingroup
\small
\linespread{1}\selectfont

\endgroup
\end{document}